\documentclass[12pt,leqno]{amsart}  

\usepackage{calc}
\usepackage[left=1in,right=1in,top=1in,bottom=1in]{geometry}  
\usepackage{graphicx}
\usepackage{amsmath,amssymb,epsfig}
\usepackage{amsthm}
\usepackage{enumerate}
\usepackage{color}
\usepackage{hyperref}

\newcommand{\Z}{{\mathbb Z}}
\newcommand{\Q}{{\mathbb Q}}

\newcommand{\too}{\longrightarrow}
\newcommand{\tn}{\textnormal}

\newcommand{\Sel}{{\mathrm{Sel}}}

\newcommand{\dimF}{{\mathrm{dim}_{\mathbb{F}_2}}}

\newcommand{\Ft}{{\mathbb{F}_2}}

\newcommand{\ord}{{ \mathrm{ord} }}
\newcommand{\prob}{{ \mathrm{Prob} }}

\newcommand{\cO}{\mathcal{O}}

\newtheorem{thm}{\bf{Theorem}}
\newtheorem{theorem}{\bf{Theorem}}[section]

\newtheorem{prop}[theorem]{\bf{Proposition}}

\theoremstyle{definition}

\newtheorem{remark}[theorem]{\bf{Remark}}

\input xy
\xyoption{all}

\usepackage[OT2,T1]{fontenc}

\DeclareSymbolFont{cyrletters}{OT2}{wncyr}{m}{n}
\DeclareMathSymbol{\Sha}{\mathalpha}{cyrletters}{"58}

\title{The Elkies Curve Has Rank 28 Subject only to GRH}

\author{Zev Klagsbrun}
\address[Z.~Klagsbrun]{Center for Communications Research, 4320 Westerra Court, San Diego, CA 92121}
\email{zdklags@ccrwest.org}

\author{Travis Sherman}
\address[T.~Sherman]{}
\email{glaisher@hotmail.com}

\author{James Weigandt}
\address[J.~Weigandt]{Institute for Computational and Experimental Research in Mathematics, Brown University, 121 South Main Street, Providence, RI 02903}
\email{james\_weigandt@brown.edu}

\begin{document}

\begin{abstract}
In 2006, Elkies presented an elliptic curve with 28 independent rational points. We prove that subject to GRH, this curve has Mordell-Weil rank equal to 28 and analytic rank at most 28. We prove similar results for a previously unpublished curve of Elkies having rank 27.

We also prove that subject to GRH, certain specific elliptic curves have Mordell-Weil ranks $20$, $21$, $22$, $23$, and $24$. This complements the work of Jonathan Bober, who proved this claim subject to both the Birch and Swinnerton-Dyer rank conjecture and GRH. This gives some new evidence that the Birch and Swinnerton-Dyer rank conjecture holds for elliptic curves over $\Q$ of very high rank.

Our results about Mordell-Weil ranks are proven by computing the $2$-ranks of class groups of cubic fields associated to these elliptic curves. As a consequence, we also succeed in proving that, subject to GRH, the class group of a particular cubic field has $2$-rank equal to $22$ and that the class group of a particular totally real cubic field has $2$-rank equal to $20$.
\end{abstract}

\maketitle

\section{Introduction}

The celebrated Mordell-Weil theorem asserts that the group $E(\Q)$ of rational points on an elliptic curve defined over $\Q$ is finitely generated. Given a particular elliptic curve $E$ over $\Q$, one can eventually find a set of generators for $E(\Q)$, but there is no algorithm known unconditionally to certify that such a set spans $E(\Q)$, or even a finite index subgroup of $E(\Q)$. It also remains unknown whether the ranks of elliptic curves over $\Q$ are uniformly bounded.

As of this writing, the highest exhibited lower bound on the rank of an elliptic curve defined over $\Q$ is due to Noam Elkies \cite{E}, who found the elliptic curve $E_{28}$ given by
\begin{multline} \label{elkies}
\resizebox{0.9\hsize}{!}{$y^2 + xy + y =x^3 - x^2 - 20067762415575526585033208209338542750930230312178956502x$}\\
 \resizebox{0.9\hsize}{!}{$+  34481611795030556467032985690390720374855944359319180361266008296291939448732243429$}
\end{multline}
together with $28$ independent rational points on $E_{28}$ with infinite order. The main result of this paper is the following:

\begin{thm}\label{mainthm} Assume the Generalized Riemann Hypothesis (GRH).
\begin{enumerate}[(i)]
	\item \label{algebraic} The Mordell-Weil group $E_{28}(\Q)$ is isomorphic to $\Bbb Z^{28}$. 
	\item \label{analytic} The analytic rank of $E_{28}$ over $\Q$ is at most $28$.
\end{enumerate}
\end{thm}


The result about the analytic rank is proved using a method of Bober \cite{BANTS} for bounding the analytic rank of an elliptic curve conditional on GRH for the Hasse-Weil $L$-function of $E_{28}$. Booker and Dwyer proved the same result using slightly more sophisticated methods, but this result has not appeared in print (see Remark 1.2 in \cite{BANTS}).

In contrast, the result about Mordell-Weil rank is proved using the classical method of $2$-descent. We show that GRH implies the dimension of the $2$-Selmer group $\textnormal{Sel}_2(E_{28}/\Q)$ is exactly $28$. This is a consequence of the following result:
\begin{thm}\label{backup}
Let $K_{28}$ be the cubic subfield of the $2$-division field of $E_{28}$. Then:
\begin{enumerate}[(i)]
	\item \label{lower} The $2$-rank of the ideal class group $Cl(K_{28})$ is at least $20$.
	\item \label{grh_upper} If GRH holds, then the $2$-rank of $Cl(K_{28})$ is exactly $20$.
\end{enumerate}
\end{thm}

Part (\ref{lower}) is obtained by applying a result of Brumer and Kramer \cite{BK}, restated as Proposition \ref{bk_bound} below, together with the lower bound on the rank of $E_{28}(\Q)$ exhibited by Elkies. Part (\ref{grh_upper}) is the result of a large class group computation described in Sections \ref{buchmann} and \ref{sec:choosing parameters}.

In addition to $E_{28}$, Elkies also shared with us the previously unpublished curve $E_{27}$ of rank at least 27 given by
\begin{multline} \label{elkies27}
\resizebox{0.9\hsize}{!}{$y^2 + xy = x^3  -55671146865244401916117773020296610079754015500970x$}\\
 \resizebox{1\hsize}{!}{$+  161981895322788558220906653027519611838007321625214218991719656790551905956.$}
\end{multline}
A list of $27$ independent points on $E_{27}$ (provided by Elkies) is included in Appendix \ref{app::gensE27}.

Using similar machinary as for $E_{28}$, we prove the following results.

\begin{thm}\label{thm:rank27thm} Assume the Generalized Riemann Hypothesis (GRH).
\begin{enumerate}[(i)]
	\item \label{algebraic27} The Mordell-Weil group $E_{27}(\Q)$ is isomorphic to $\Bbb Z^{27}$. 
	\item \label{analytic27} The analytic rank of $E_{27}$ over $\Q$ is at most $27$.
\end{enumerate}
\end{thm}

\begin{thm}\label{thm:rank27classgrpthm}
Let $K_{27}$ be the cubic subfield of the $2$-division field of $E_{27}$. 
\begin{enumerate}[(i)]
	\item \label{lower27} The $2$-rank of the ideal class group $Cl(K_{27})$ is at least $22$.
	\item \label{grh_upper27} If GRH holds, then the $2$-rank of $Cl(K_{27})$ is exactly $22$.
\end{enumerate}
\end{thm}

To our knowledge, the $2$-rank of $Cl(K_{27})$ is the largest known for a cubic field to have been proven under standard hypotheses. Similarly, the $2$-rank of $Cl(K_{28})$ is the largest known for a totally real cubic field proven under standard hypotheses. 


We also include the following result concerning previous record holding elliptic curves.
\begin{thm}\label{other_curves}Subject to GRH:
\begin{enumerate}[(i)]
\item $E_{20}:y^2 + xy = x^3 - 431092980766333677958362095891166x  $ \\ 

\vspace{-0.2in}
\hspace{1in} $+ 5156283555366643659035652799871176909391533088196$
 has rank 20.
\item$E_{21}:y^2 + xy + y = x^3 + x^2  -215843772422443922015169952702159835x$ \\ 

\vspace{-0.2in}
\hspace{0.65in} $-19474361277787151947255961435459054151501792241320535$  has rank 21.
\item $E_{22}:y^2 + xy + y = x^3 - 940299517776391362903023121165864$ \\ 

\vspace{-0.2in}
\hspace{0.9in} $ + 10707363070719743033425295515449274534651125011362$ has rank 22.
\item $E_{23}:y^2 + xy + y = x^3 - 19252966408674012828065964616418441723$ \\ 

\vspace{-0.2in}
\hspace{0.4in} $ + 32685500727716376257923347071452044295907443056345614006$ has rank 23.
\item $E_{24}:y^2 + xy + y = x^3 - 120039822036992245303534619191166796374 $ \\ 

\vspace{-0.2in}
\hspace{0.3in} $+ 504224992484910670010801799168082726759443756222911415116$
 has rank 24.
\end{enumerate}
\end{thm}

Each of these curves $E_r$ was the first exhibited elliptic curve over $\Q$ with Mordell-Weil rank at least $r$ \cite{DujellaHistory}. For each $r \in \{20,21,22,23,24\}$, Bober proved that GRH for the Hasse-Weil $L$-function of $E_r$ implies the analytic rank of $E_{r}$ is at most $r$ \cite{BANTS}. Theorem \ref{other_curves} was obtained by computing the $2$-rank of the class group of the cubic subfield of the two-division field $\Q(E_r[2])$ for each $E_r$ and then applying Proposition \ref{bk_bound}. This is summarized in Appendix \ref{proof_other_curves}.

\begin{remark}\label{rem::GRHRemark}
The results in this paper depend on GRH in two different ways. Part \ref{analytic} of Theorem \ref{mainthm} requires GRH for the L-function $L(s,E_{28})$ of the elliptic curve $E_{28}$. In all other instances, we use GRH as described in Section \ref{subsec:fbbounds} to assert that $Cl(K)/Cl(K)^2$ is generated by primes below a particular bound. We therefore need to assert GRH for the zeta functions of a large but finite number of unramified quadratic extensions of $K$.
\end{remark}

\subsection{Data}

Our computations use a variant of Buchmann's algorithm which is similar to the number field sieve. As detailed in Section \ref{buchmann}, this algorithm proceeds by collecting relations for $Cl(K)/Cl(K)^2$ supported on primes below a certain bound. We have made these relations available at \url{https://github.com/zevklagsbrun/ElkiesCurve}, so the enterprising reader can verify our results.

\subsection{Acknowledgements}

We would like to express our thanks to Noam Elkies for sharing the curve $E_{27}$ with us and for providing a number of helpful suggestions along the way. We would also like to thank Jonathan Bober for sharing the results of Booker and Dwyer with us.

\section{Bounding Analytic Ranks} 

One way to get conditional upper bounds on the Mordell-Weil rank of an elliptic curve over $\Q$ involves the study of the Hasse--Weil $L$-function $L(s,E)$. 

The modularity theorem \cite{BCDT} relates $L(s,E)$ to the Mellin transform of a certain weight $2$ modular form for $\Gamma_0(N)$ where $N = N(E)$ is the conductor of $E$. This implies $L(s,E)$ has an analytic continuation and that the completed $L$-function $\Lambda(s,E) = 2N^{s/2} (2 \pi)^{-s}\Gamma(s)L(s,E)$ satisfies the functional equation $\Lambda(s,E) = \epsilon(E) \Lambda(1-s,E)$ where $\epsilon(E) \in \{ \pm 1 \}$ is the global root number of $E$. Note that $L(s,E)$ is the analytically normalized $L$-function, so that the functional equation is symmetric about the critical line $\Re(s) = 1/2$. 

Define the analytic rank of $E$ by $r_{\textnormal{an}}(E) = \textnormal{ord}_{s = 1/2} L(E,s)$. The functional equation implies $\epsilon(E) = (-1)^{r_{\textnormal{an}}(E)}$. The Birch and Swinnerton-Dyer rank conjecture (BSD) asserts that $r_{\textnormal{an}}(E)$ is equal to the rank of the Mordell--Weil group $E(\Q)$. 
%


In \cite{BANTS}, Bober describes a way to compute conditional upper bounds on $r_{\textnormal{an}}(E)$ subject to GRH for $L(s,E)$. The main tool used is the explicit formula  \cite[Lemma 2.1]{BANTS} expressing the $\sum_\gamma f(\gamma)$ of values of a test function $f$ as $s = 1/2 + i \gamma$ ranges over the zeros of $L(s,E)$. This is a natural analogue of Weil's formulation of the Riemann-von Mangoldt formula for $\zeta(s)$. The terms in the explicit formula are easy to compute for test functions chosen from the parametrized Fej\'er kernel
\begin{equation}\label{test_function} 
f_{\Delta}(x) = \left ( \dfrac{\sin(\Delta \pi x)}{\Delta  \pi x}\right )^2
\end{equation}
for real numbers $\Delta \geq 1$. Since $f_\Delta(0) = 1$, we have 
\begin{equation}
	\sum_\gamma f_\Delta(\gamma) = r_{\textnormal{an}}(E) + \sum_{\gamma \neq 0} f_\Delta(\gamma).
\end{equation}
If $s = 1/2 + i \gamma$ satisfies $\gamma \in \Bbb R$, then $f_{\Delta}(\gamma) \geq 0$. Thus, if the GRH holds for $L(s,E)$, then $\sum_{\gamma \neq 0} f_\Delta(\gamma) \geq 0$ and hence $\sum_{\gamma} f_{\Delta}(\gamma)$ is a conditional upper bound for $r_{\textnormal{an}}(E)$.
The explicit formula lets us express this bound as a sum of three more easily understood terms.

\begin{prop}[Bober \cite{BANTS}] Let $E$ be an elliptic curve over $\Q$ of conductor $N(E)$ and let $\Delta \geq 1$ be a real number. For each prime $p$, choose a factorization
$$L_p(s,E) = (1-\alpha(p)p^{-s})(1-\beta(p)p^{-s})$$ 
in terms of the Euler product $L(s,E) = \prod_{p} L_p(s,E)^{-1}$ for $\Re(s) > 1$. Define the arithmetic term\begin{equation}
g_{\textnormal{an}}(\Delta,E)  =  - \dfrac{1}{\Delta \pi} \sum_{p \leq \exp(2 \Delta \pi)} \log p \sum_{k = 1}^{\lfloor{2 \Delta \pi / \log p}\rfloor}\frac{k}{p^{k/2}} \bigl (\alpha(p)^k + \beta(p)^k \bigr)\left(1 - \dfrac{\log p^{k/2} }{\Delta \pi} \right),
\end{equation}
the archimedean term
\begin{equation} 
u_{\textnormal{an}}(\Delta)  = \dfrac{1}{\pi} \Re \left \{ \int_{-\infty}^{\infty} \dfrac{\Gamma'}{\Gamma}(1 + it)f_{\Delta}(t) dt \right \},
\end{equation}
and the conductor term
\begin{equation}
n_{\textnormal{an}}(\Delta,E)  = \dfrac{1}{\Delta \pi} \log \dfrac{\sqrt{N(E)}}{2\pi}.
\end{equation}
As $s = 1/2 + i\gamma$ ranges over the zeros of $L(s,E)$, we have
\begin{equation}\label{explicit}
\sum_\gamma f_\Delta(\gamma) = g_{\textnormal{an}}(\Delta,E)  + u_{\textnormal{an}}(\Delta) + n_{\textnormal{an}}(\Delta,E).
\end{equation}
\end{prop}
\begin{proof} This is \cite[Equation 3]{BANTS}
\end{proof}
%

Each term on the righthand side of equation (\ref{explicit}) can be computed to high precision. To compute the arithmetic term, one must compute the local factors $L_p(s,E)$ for $p \leq \exp(2 \Delta \pi)$. For $\Delta \leq 4.41$, this can be done efficiently with Andrew Sutherland's \texttt{smalljac} package. 
Simon Spicer observed that the archimedean term has the closed form
\begin{equation}
u_{\textnormal{an}}(\Delta) = -\dfrac{\eta}{\pi^2 \Delta} + \dfrac{1}{2 \pi^3 \Delta^2} \left ( \dfrac{\pi^2}{6} - \textnormal{dilog}(e^{-2\pi \Delta}) \right ),
\end{equation}
in terms of the Euler-Mascheroni constant $\eta \approx 0.57721566$ and the dilogarithm function $\textnormal{dilog}(x) = \sum_{n \geq 1} {x^n}/{n^2}$. 
We may therefore effectively compute upper bounds for $r_{\textnormal{an}}(E)$ conditional only on GRH for $L(s,E)$. This functionality is implemented in the \texttt{SageMath} \cite{SRM} software package via the command \texttt{analytic\_rank\_upper\_bound}.
\begin{proof}[Proof of part (\ref{analytic}) of Theorem \ref{mainthm}:]
Using the Sage (version 7.1) implementation of Bober's method with a tightness parameter of $\Delta = 4$, we found that $r_\tn{an}(E_{28}) \le \sum_\gamma f_4(\gamma) < 30$. Since $E_{28}$ has root number $+1$, $r_\tn{an}(E_{28})$ must be even and as a result, $E_{28}$ has analytic rank at most 28. This computation took approximately 40 hours on an Intel i7 processor. (We found that the runtime was significantly improved by passing a list of bad primes of $E_{28}$ to \texttt{analytic\_rank\_upper\_bound} and setting the flag \texttt{adaptive} to \texttt{false}.)
%
\end{proof}

\begin{proof}[Proof of part (\ref{analytic27}) of Theorem \ref{thm:rank27thm}:]
We used Bober's method as above but with a tightness parameter of $\Delta = 3.1$. We found that $r_\tn{an}(E_{27}) \le \sum_\gamma f_{3.1}(\gamma) < 29$. Since $E_{27}$ has root number $-1$, $r_\tn{an}(E_{27})$ must be odd and as a result, $E_{27}$ has analytic rank at most 27. This computation took approximately three minutes on an Intel i7 processor.
%
%
\end{proof}

\section{The 2-Selmer Group}\label{selgrp}

One of the most common methods for obtaining an upper bound on the Mordell-Weil rank of $E$ is studying the $2$-Selmer group $\Sel_2(E/\Q)$ of $E$. 
We briefly recall the definition and some important properties here and refer the reader to Chapter X of \cite{AEC} for a more details.

If $E$ is an elliptic curve defined over $\Q$, then $E(\Q)/2E(\Q)$ maps into $H^1(\Q, E[2])$ via the Kummer map $\delta$. The following diagram commutes for every place $v$ of $\Q$, where $\delta_v$ is the local Kummer map.
$$\centerline{
\begin{xy}
\xymatrix{%
E(\Q)/2E(\Q)  \ar[d] \ar[r]^{\delta} & H^1(\Q, E[2]) \ar[d]^{\textnormal{Res}_v} \\
E(\Q_v)/2E(\Q_v)   \ar[r]^{\delta_v} & H^1(\Q_v, E[2])}
\end{xy}
}$$
The \textbf{2-Selmer group} of $E/\Q$, denoted $\Sel_2(E/\Q)$, is defined as
\begin{equation*}
\Sel_2(E/\Q) = \left \{c \in H^1(\Q, E[2]) : \textnormal{Res}_v(c) \in \delta_v\left (E(\Q_v)/2E(\Q_v) \right) \text{ for all } v \text{ of } \Q \right \}.
\end{equation*}
This group has the structure of a finite dimensional $\Ft$ vector space and it sits in the exact sequence
\begin{equation*}
0 \too E(\Q)/2E(\Q) \too \Sel_2(E/\Q) \too \Sha(E/\Q)[2] \too 0,
\end{equation*}
where $\Sha(E/\Q)$ is the Tate-Shafarevich group of $E$. It follows that the rank of $E(\Q)$ is at most $\dimF \Sel_2(E/\Q) - \dimF E(\Q)[2]$. Unlike the rank of $E(\Q)$, $\dimF \Sel_2(E/\Q)$ is known to be computable. Computing $\textnormal{Sel}_2(E/\Q)$ gives an upper bound on the rank of $E(\Q)$. Often $\Sha(E/\Q)[2]$ is trivial, in which case this bound is sharp.

\subsection{The Brumer-Kramer Bound}

In \cite{BK}, Brumer and Kramer study the structure of the cohomology group $H^1(\Q,E[2])$ and of the images of $E(\Q_v)/2E(\Q_v)$ in $H^1(\Q_v,E[2])$ when $E(\Q)[2] = 0$. By doing so, they obtain an upper bound on $\dimF \Sel_2(E/\Q)$ in terms of the $2$-rank of the class group of the cubic subfield of $\Q(E[2])$ and information about the places where $E$ has bad reduction. To state their result, we need to first introduce some notation.

Let $\varDelta$ be the discriminant of $E$ and set $\Phi_m$ to be the set of primes $p$ at which $E$ has multiplicative reduction and $\textnormal{ord}_p\varDelta$ is even. Set $\Phi_a$ to be the set of primes at which $E$ has additive reduction, and for each $p \in \Phi_a$, let $n_p$ be the number of primes of $K$ lying above $p$, where $K$ is the cubic subfield of the $2$-division field of $E$. We then define $$g(E) = \dimF Cl(K)[2], \quad u(E) = \begin{cases} 1 & \textnormal{if $\varDelta < 0$,} \\ 2 & \textnormal{if $\varDelta > 0$,} \end{cases} \quad \text{and} \quad n(E) = \#\Phi_m + \sum_{p \in \Phi_a} (n_p - 1).$$

\begin{prop} [Brumer and Kramer \cite{BK}] \label{bk_bound} With notation as above, we have
\begin{equation} \label{first_majorization}
\dimF \textnormal{Sel}_2(E/\Q) 
						\leq g(E) + u(E) + n(E).
\end{equation}
\end{prop} 
\begin{proof} This is \cite[Prop. 7.1]{BK}.
\end{proof}
We are now is a position to prove the lower bound from Theorem \ref{backup}.
\begin{proof}[Proof of Theorem \ref{backup} - part (\ref{lower})] Since the rank of $E_{28}(\Q)$ is at least $28$, applying Proposition 3.1 to $E_{28}$ gives
\begin{equation} \label{bk_elkies}
28 \leq \dimF \textnormal{Sel}_2(E_{28}/\Q) \leq g(E_{28}) + u(E_{28}) + n(E_{28}).
\end{equation}
The arithmetic term $g(E_{28})$ is the $2$-rank of the ideal class group of the cubic subfield $K_{28}$ of $\Q(E_{28}[2])$. Since $\varDelta(E_{28}) > 0$, we have $u(E_{28}) = 2$. Computing local information about $E_{28}$, we find that $\Phi_m = \{5,7,11,13\}$, $\Phi_a = \{3\}$, and the prime $3$ splits completely in $K_{28}$. This gives the conductor term $n(E_{28}) = 4 + (3 - 1) = 6$. Combining this information with inequality (\ref{bk_elkies}), we get $\dimF Cl(K_{28})[2] \geq 28 - u(E_{28}) - n(E_{28}) = 20$.
\end{proof}

\section{An Algorithm for Computing $\dimF Cl(K)[2]$}\label{buchmann}

To compute an upper bound on $\dimF \Sel_2(E_{28}/\Q)$, we need to bound $g(E_{28}) = \dimF Cl(K_{28})[2]$. The method we use to bound $\dimF Cl(K_{28})[2]$ is based on an algorithm of Buchmann et al. in \cite{BJNTW} that is inspired by the number field sieve. While Buchmann's algorithm is able to compute the exact structure of $Cl(K)$ subject to GRH for a general number field $K$, we are able to take a few shortcuts that simplify the algorithm since $K_{28}$ is a cubic field and because we are only concerned with $\dimF {Cl}(K_{28})[2]$. We describe our variant of Buchmann's algorithm below.


\subsection{A Presentation for $Cl(K)$}\label{subsec::presentation}

We start with a factor base $\mathcal{P}$ of degree one prime ideals of $\cO_K$ (including ramified prime ideals with residue class field degree equal to one) with norm less than a bound $\mathcal{B}$. The factor base $\mathcal{P}$ will serve as a generating set for the class group $Cl(K)$. 

To compute a presentation for $Cl(K)$, we need relations supported on $\mathcal{P}$. Relations are given by principal ideals $(\beta)$ such that $\mathbf{N}_{K/\Q} \beta$ is $\mathcal{B}$-smooth and $(\beta)$ factors as a product of primes in $\mathcal P$.

Factoring these relations as $$(\beta) = \prod_{\frak p \in \mathcal{P}} \frak p^{\textnormal{ord}_{\frak p}(\beta)},$$ we obtain a matrix $M$ with entries in $\Z$. Assuming that $\mathcal{P}$ is large enough and $M$ contains enough relations, the structure of $Cl(K)$ can then be read off from the Hermite normal form (HNF) of $M$.



Computing the HNF of a large matrix is difficult because it requires doing a large integral linear algebra computation. However, since we are only interested in computing the size of $Cl(K)[2] \simeq Cl(K)/Cl(K)^2$, we can take the coefficients of this matrix to be in $\mathbb{F}_2$ instead. In this case, the dimension of the right nullspace of this $\Bbb F_2$-matrix is an upper bound for the dimension of the subspace of $Cl(K)/Cl(K)^2$ generated by the primes in $\mathcal P$.

\subsection{The Size of the Factor Base $\mathcal{P}$}\label{subsec:fbbounds}

By a result of Bach \cite{Bach}, if GRH holds, then $Cl(K)$ is generated by the primes of $K$ with norm less than $12(\log \mathfrak{d}(K))^2$ (the ``Bach bound'') where $\mathfrak{d}(K)$ is the discriminant of $\cO_K$. Subsequent work by Belabas, Diaz y Diaz, and Friedman \cite{BDyDF} gives an alternative and less explicit bound $B_K$ (the ``Belabas bound'') such that if GRH holds, then $Cl(K)$ is generated by the primes of $K$ with norm less than $B_K$. 

While the Belabas bound is asymptotically worse than the Bach bound, it is often quite a bit smaller than the Bach bound for fields of interest. We will therefore use the term ``GRH bound'' to refer to the smaller of the Bach bound and the Belabas bound for a particular field $K$. As long as $\mathcal P$ contains all the primes of $K$ with norm less than the GRH bound, the rank of the nullspace of the relation matrix is an upper bound for $\dimF Cl(K)[2]$ under GRH.

It is easy to see that if $K$ is a cubic field, then the same result holds if $\mathcal P$ only contains all degree one primes of norm less than the GRH bound. If $\mathfrak{p}$ is a degree 3 prime, then $\mathfrak{p}$ is automatically principal and need not be included in $\mathcal{P}$.  If $\mathfrak{p}$ is a degree two prime of $\cO_K$ lying above a rational prime $p$, then there is a degree one prime $\mathfrak{p}^\prime$ such that $\mathfrak{p}\mathfrak{p}^\prime = (p)$. Since the ideal classes $[\mathfrak{p}]$ and $[\mathfrak{p}^\prime]$ are inverses of each other in $Cl(K)$, it suffices to include only $\mathfrak{p}^\prime$ in $\mathcal P$.

\begin{remark}
Both the Bach and Belabas bounds require GRH to hold for the zeta functions of all unramified abelian extensions of $K$. However, since we are only concerned with $Cl(K)[2]$, it suffices to assume GRH for the zeta functions of all unramified quadratic extensions of $K$.
\end{remark}

\subsection{Provable Lower Bounds}\label{subsec::provablelowerbounds}

Having chosen a suitably large factor base $\mathcal{P}$ as described in Section \ref{subsec:fbbounds}, the presentation in Section \ref{subsec::presentation} can be used to get an upper bound on $\dimF Cl(K)[2]$ subject to GRH. We now describe a method to use the presentation matrix $M$ described in Section \ref{subsec::presentation} to produce an unconditional lower bound for $\dimF Cl(K)[2]$. In the event that this unconditional lower bound matches the conditional upper bound, we obtain an exact value for $\dimF Cl(K)[2]$ subject to GRH.
 
Following Section 5 of \cite{C}, we define the 2-Selmer group of the field $K$, denoted $\Sel_2(K)$ by $$\Sel_2(K) = \{ \beta \in K^\times/(K^\times)^2 : \ord_\mathfrak{p} \beta \equiv 0 \pmod{2} \text{ for all primes } \mathfrak{p} \text{ of } K\}.$$ It is easy to see that left nullvectors of the presentation matrix $M$ in Section \ref{subsec::presentation} yield representatives of elements in $\Sel_2(K)$.

The Selmer group $\Sel_2(K)$ sits in the exact sequence $$0 \too \cO_K^\times/(\cO_K^\times)^2 \too \Sel_2(K) \too Cl(K)[2] \too 0.$$ The left hand term $\cO_K^\times/(\cO_K^\times)^2$ is known to be isomorphic to $(\Z/2\Z)^{r_1 + r_2}$, and as a result, we have $\dimF Cl(K)[2] = \dimF \Sel_2(K) - (r_1 + r_2).$ Therefore, by producing $r$ independent elements of $\Sel_2(K)$, we are able to prove that $\dimF Cl(K)[2] \ge r - (r_1+ r_2)$.

As noted above, left nullvectors of the presentation matrix $M$ give elements of $\Sel_2(K)$. Supposing that we have found $r$ such elements $\beta_1$, $\beta_2$, \ldots, $\beta_r$, we need to show that they represent independent elements of $K^\times/(K^\times)^2$. This can be accomplished by finding prime ideals $\mathfrak{p}_1, \mathfrak{p}_2, \ldots, \mathfrak{p}_r \not \in \mathcal{P}$ such that the matrix $\left(\chi_{\mathfrak{p}_i}(\beta_j)\right)$ has rank $r$, where $\chi_{\mathfrak{p}_i}$ is the additive Legendre character on $\cO_K/\mathfrak{p}_i$.

\subsection{Constructing Relations}

While the smooth relations $(\beta)$ described in Section \ref{subsec::presentation} need not be constructed in any particular way, there is a computationally efficient method for constructing them based on the number field sieve factoring algorithm \cite{NFS}. We now describe the basic idea behind sieving.

Letting $f(x)$ be any defining polynomial for $K$ and $\alpha$ be a root of $f(x)$. Suppose that $\mathfrak{p}$ is a prime ideal of $\cO_K$ given by a rational prime $p$ and a root $r$ of $f(x) \pmod p$. We can then see that $\mathfrak{p}$ divides $a + b\alpha$ if $r \equiv -ab^{-1} \pmod p$. Therefore, we may identify all $a + b\alpha$ in a large range $-A \le a \le A$ and $1 \le b \le B$ such that $\mathfrak{p}$ divides $a + b\alpha$.

Doing this for all of the primes $\mathfrak{p} \in \mathcal{P}$, we may identify all $a + b\alpha$ with $-A \le a \le A$ and $1 \le b \le B$ such that $a + b\alpha$ is divisible by many different primes $\mathfrak{p}$ in $\mathcal{P}$ and therefore more likely to factor completely in terms of primes in $\mathcal{P}$.  After identifying many candidate $a + b\alpha$, we may apply trial division (or any other factoring algorithm) to discover which $(a + b\alpha)$ factor entirely in $\mathcal{P}$. 

\section{Choosing Parameters}\label{sec:choosing parameters}

Constructing relations requires choosing three parameters: the polynomial $f(x)$ defining $K$, a factor base bound $\mathcal{B}$, and a sieve region $[-A,A] \times [1,B]$. We now describe how to choose these parameters with a focus on the field $K_{28}$.

\subsection{Choosing the Polynomial}\label{subsec::polychoice}

If $K$ is a cubic field, then we can always find a cubic polynomial $f(x)$ defining $K$ such that the discriminant $\Delta(f)$ of $f(x)$ is equal to the discriminant $\Delta_{\cO_K}$ of $\cO_K$. This polynomial $f(x)$ is unique up to the action of $\mathrm{GL}_2(\mathbb{Z})$. By applying Julia reduction to $f(x)$ (see \cite[Algorithm 1]{Cre}), we can obtain what is in some sense the smallest polynomial defining $K$.

For the field $K_{28}$, this reduced polynomial is given by
\begin{multline}\label{Kpoly}
f(x) = 64023127168000x^3 + 10309553525987840512490787747x^2 \\ - 3858878002265332645698861066081585182608x \\ - 
    69043295714402138353376748510210837676894689434302674.
\end{multline}

\subsection{Choosing the Factor Base Bound}

Section \ref{subsec:fbbounds} addresses how small the factor base bound $\mathcal{B}$ may be. However, choosing the smallest possible $\mathcal{B}$ makes it less likely that an element $a + b\alpha$ of a given size will be smooth. While choosing a larger bound $\mathcal{B}$ will make it easier to find relations, it will make follow-on linear algebra work harder since the size of the matrix $M$ will increase. Our primary goal in choosing a factor base bound was that the resulting matrix could be processed in \textbf{magma}. For $K_{28}$, a natural touchstone was the Bach bound of $1,202,639$, which gave us a factor base $\mathcal{P}$ containing $93,121$ primes.


\subsection{Choosing a Sieve Region}\label{subsec::chooseregion}

In order to choose a sieve region $\mathcal{A} = [-A,A] \times [1,B]$, we need to consider two things - how large our sieve region should be (that is, $2 \cdot A \cdot B$) and how skew that region should be (that is, $A/B$).

\subsubsection{Skewness}\label{subsubsec:skewness}

Let $F(X,Y)$ be the homogenization of $f(x)$.  The norm $\mathbf{Norm} (a + b\alpha)$ is given by $$\mathbf{Norm} (a + b\alpha) = N_{K/\Q} (a + b\alpha) = \dfrac{F(a,-b)}{c_3(f)},$$ where $c_3(f)$ is the leading coefficient of $f(x)$. Assuming that all primes dividing $c_3(f)$ are in $\mathcal{P}$ and that $\mathcal{P}$ contains all of the primes dividing $(\alpha)$, then $(a+ b\alpha)$ factors completely in $\mathcal{P}$ if and only if $F(a,-b)$ is $\mathcal{B}$-smooth. Therefore, the likelihood that the ideal generated by $a + b\alpha$ for a random $(a,b)$ in $\mathcal{A}$ factors completely in $\mathcal{P}$ is given by the probability that $F(a,-b)$ is $\mathcal{B}$-smooth for a random $(a,b) \in \mathcal{A}$. To first approximation, this probability is determined by the size of $|F(a,-b)|$.

Rather than attempt to understand how the size of $|F(a,-b)|$ is distributed on $\mathcal{A}$, we may simply consider the maximum of $|F(a,-b)|$ on the boundary of $\mathcal{A}$. To do so, we consider the individual terms of $F(X,Y)$, which attain their maximum (in absolute value) of $|c_i|A^iB^{3-i}$ at the point $(A,B)$, where $c_i$ is the coefficient of $x^i$ in $f(x)$. An ideal skewness would have $A/B$ chosen so that each $|c_i|A^iB^{3-i}$ was of roughly equal size.

In our case, the polynomial $F(X,Y)$ does not admit a skewness such that each $|c_i|A^iB^{3-i}$ is of roughly equal size. We may however choose a skewness so that the largest two values of $|c_i|A^iB^{3-i}$ are roughly the same. For our polynomial $F(X,Y)$, this suggests a skewness of $s  = 2^{41.25}$.

If $\mathcal{A}$ has skewness  $s = 2^{41.25}$, then $\mathcal{A}$ must have an area of at least $S = 42.25$ bits in order to have integral points with $b \ne 0$. We will therefore assume that $S$ is at least $42.25$ bits. In this case, we find that the two largest values of $|c_i|A^iB^{3-i}$ are $|c_2|A^2B$ and $|c_0|B^3$, which both have size $175.5 + \frac{3}{2}(S - 42.25)$ bits. As the values of  $|c_1|AB^2$ and $|c_3|A^3$ are substantially smaller, we may approximate the maximum of $|F(X,-Y)|$ on $\mathcal{A}$ as $|c_2|A^2B + |c_0|B^3$ which is roughly $176.5 + \frac{3}{2}(S - 42.25)$ bits in size.

\subsubsection{Smoothness Probabilities}

We now must consider how large of a sieve region to use. In order to produce enough relations, we need 
$$\frac{1}{\zeta(2)}2\cdot A \cdot B \cdot \prob\left ( F(a,-b) \text{ is $\mathcal{B}$-smooth } \mid (a,b) \in \mathcal{A} \right) \ge |\mathcal{P}|.$$
We therefore need to estimate the probability that $F(a,-b)$ is $\mathcal{B}$-smooth when $(a,b)$ is chosen randomly from $\mathcal{A}$.

Let $\rho(u)$ denote Dickman's rho function. If $n$ is a random number of size $C$, then standard results tell us that the probability that $n$ is $\mathcal{B}$-smooth can be approximated by $\rho\left (\frac{\log C}{\log \mathcal{B}}\right)$ as long as $\mathcal{B} \ge (\log C)^{2+\epsilon}$ (assuming GRH) \cite{H}. However, if $n = F(a,-b)$ is a random value of $F(X,Y)$, then the probability that $n$ is a $\mathcal{B}$-smooth is affected by a parameter known as $\alpha = \alpha(F)$ which takes into account the modular root properties of $F(X,Y)$ \cite{Murphy}. Assuming that $n$ has size $C$, the probability that $n$ is $\mathcal{B}$-smooth is equal to the probability that a random number of size $C\alpha$ is smooth. We may therefore approximate the probability that $n$ is smooth by $\rho\left (\frac{\log C + \log \alpha}{\log \mathcal{B}}\right)$. For our polynomial $F(X,Y)$, \textbf{magma} tells us that $\alpha \approx -2^{1.9}$.

\subsubsection{Relation Estimates}\label{subsec::relest}

The size of $|F(a,-b)|$ for $(a,b) \in \mathcal{A}$ may vary considerably. One very crude estimate for a representative value of $|F(a,-b)|$ would be the maximum $|F(X,Y)|$ on the boundary of $\mathcal{A}$, which we calculated to be $176.5 + \frac{3}{2}(S - 42.25)$ bits at the end of Section \ref{subsubsec:skewness}. A somewhat less crude estimate would be given by decomposing $\mathcal{A}$ into shells and taking the maximum of $|F(X,Y)|$ on each shell to be representative of the values of $|F(X,Y)|$ on that shell.

Assuming that $S = 42.25 + 0.25\cdot k$, then using shells of radius $0.25$, we estimate that the number of relations for a sieve region of size $S$ is given by $$\frac{1}{\zeta(2)} \sum_{i = 0}^k  
\beta(k) \rho\left(\frac{176.5 + \alpha + \frac{3}{2}(0.25k)}{\log_2 \mathcal{B}} \right) = \frac{1}{\zeta(2)} \sum_{i = 0}^k \beta(k) \rho\left(\frac{174.6 + \frac{3k}{8}}{20.2} \right),$$ where $\beta(k) =  \begin{cases} 2^{42.25} & \textnormal{if $k = 0$} \\ 2^{42.25 + 0.25k} - 2^{42.25 + 0.25(k-1)} & \textnormal{if $k > 0$}\end{cases}.$

Estimates for several values of $S$ are given Table \ref{tab:tau_s}. Since $\mathcal{P}$ consists of $93121$ primes, Table \ref{tab:tau_s} suggests that $\mathcal{A}$ should have size somewhere between $2^{46}$ and $2^{46.5}$.

\begin{table}[h]
\renewcommand{\arraystretch}{1.25}
\begin{tabular}{|c|c|c|}
\hline
$S $ & Estimated number of relations \\
\hline
45 & 51,394\\
\hline
45.5 & 65,320\\
\hline
46 &82,602\\
\hline
46.5 &104,046\\
\hline
47 & 130,648\\
\hline
47.5 &163,641\\
\hline
48 & 204,554\\
\hline
48.5 & 255,278\\
\hline
\end{tabular}
\vspace{0.2in}
\caption{Estimated numbers of relations for different sieve regions}
\label{tab:tau_s}
\end{table}

\subsection{The Computation for $K_{28}$}


We chose to sieve the region $[-2^{43.75},2^{43.75}] \times [1,5]$, which has size roughly $2^{47}$. We found $133,637$ relations, which is in line with the prediction in Table~\ref{tab:tau_s}.  We were able to augment these with $15,518$ relations coming from rational primes $p$ that split completely in $\cO_K$ (that is, relations of the form $p + 0\cdot \alpha$).

Unsurprisingly, there were a small number of primes in $\mathcal{P}$ that did not appear in any relation. However, all of these primes had norm greater than the Belabas bound of $200,439$, so we were able to safely remove them from $\mathcal{P}$.

However, when we reduced the entries of the relation matrix into $\mathbb{F}_2$, we discovered that the columns for the degree one primes $\mathfrak{p}_7$ and $\mathfrak{p}_{13}$ above $7$ and $13$ were identically zero. Further inspection revealed that since $\ord_{\mathfrak{p}_7} \alpha = -2$, we had $\ord_{\mathfrak{p}_7} a + b \alpha = -2$ for all relations $a + b\alpha$. The same held true for $\mathfrak{p}_{13}$. We were able to remedy this by sieving for a small number of relations of the form $a + 7\alpha$ and $a+ 13\alpha$ and including these. (A degree one prime above $17$ would have exhibited the same phenomenon had we not included the rational relation $17 + 0\cdot \alpha$.)

Upon computing the right nullspace of our relation matrix, we discovered that there were a small number of low-weight vectors that seemed spurious. These corresponded to primes (all above the Belabas bound) that did not appear in enough relations. By removing the relations incident on these primes, we were able to remove the primes from our factor base. A second nullspace computation showed that the nullity of the modified relation matrix was in fact 20, proving  Part (\ref{grh_upper}) of Theorem \ref{backup}: if GRH holds, then the 2-rank of $Cl(K_{28})$ is exactly $20$.

The dominant portion of the computation was the sieving step. Since the NFS functionality built into \textbf{magma} did not support our chosen sieve region, we wrote speciality C code to handle the sieving. The sieve portion of the computation took roughly 14.5 core days on a cluster composed of Intel 2.6 GHz processors. The linear algebra portion of the computation was completed on a single instance of \textbf{magma} running on a desktop. This portion of the computation took roughly 15 minutes and used under 16 GB of memory. 
Now we prove the remaining part of Theorem \ref{mainthm}.
\begin{proof}[Proof of Theorem \ref{mainthm} - part (\ref{algebraic})]
%
We have shown that if GRH holds, then $g(E_{28}) = \dimF Cl(K_{28})[2]$ is at most $20$. Combining Proposition \ref{bk_bound} with Elkies's lower bound on the rank of $E_{28}(\Q)$ and part (\ref{grh_upper}) of Theorem \ref{backup}, we get 
$$ 28 \leq \textnormal{rank}\, E_{28}(\Q) \leq  \dimF \textnormal{Sel}_2(E/\Q) \leq g(E_{28}) + u(E_{28}) + n(E_{28}) = 20 + 6 + 2 = 28.$$
Hence GRH implies the rank of $E_{28}(\Q)$ is exactly $28$. Since $E_{28}(\Q)_{\textnormal{tors}}$ is trivial, we conclude that $E_{28}(\Q) \simeq \Z^{28}$ subject to GRH.
\end{proof}

\subsection{Computation for $K_{27}$}

We used the same considerations described in Sections \ref{subsec::polychoice} - \ref{subsec::chooseregion} to choose parameters for $K_{27}$.

The appropriately minimized and reduced polynomial for $K_{27}$ is given by
\begin{multline}\label{K27poly}
f(x) = 15560036076469248x^3 + 51468441407469319836143473x^2 \\ - 497312227802505407769400165687028x \\+ 556884612253557846953628131195272740623601.
\end{multline}
The relative size of the coefficients of $f(x)$ suggest that we should use a skewness of $s = 2^{26.625}$.

To choose the factor base bound $\mathcal{B}$, we first considered the Belabas bound $\mathcal{B}_B$ which \textbf{magma} says is equal to $143,829$. However, a back of the envelope calculation showed that finding relations with this bound would be particularly difficult, and that choosing $\mathcal{B} = 4\cdot \mathcal{B}_B$ would be more effective. This resulted in a factor base with $47,063$ primes.
 
Using the method described in Section \ref{subsec::relest} for estimating relations, we settled on the sieve region $[-2^{34},2^{34}] \times [1,166]$ which has size roughly $2^{42.375}$. Sieving this region yielded $54,597$ relations which we augmented with an additional $7,817$ relations coming from rational primes.

As was the case for $K_{28}$, the initial right nullspace computation produced a small number of low-weight vectors. After removing the corresponding columns and the rows incident on them (as well as the empty columns), we were left with a $62,370 \times 46,513$ matrix $M$. As all of the columns removed corresponded to primes above the Belabas bound, this did not affect the integrity of our computation. A computation in \textbf{magma} then showed that the right nullspace of $M$ had dimension $22$, and as a result $Cl(K_{27}) \le 22$ subject to GRH.

Taking the submatrix of $M$ consisting of the first $45,325$ rows, we obtained a matrix with a 32-dimensional left nullspace. Using the technique described at the end of Section \ref{subsec::provablelowerbounds}, we were able to show that this nullspace contained $24$ independent elements of $K^\times/(K^\times)^2$. As $K_{27}$ has one real and one complex place, this proves that $\dimF Cl(K_{27})[2] \ge 22$ unconditionally. Combined with the upper bound computed above, we therefore get that  $Cl(K_{27}) \le 22$ subject to GRH. This proves Theorem \ref{thm:rank27classgrpthm}.

Now we prove the remaining part of Theorem \ref{thm:rank27thm}.
\begin{proof}[Proof of Theorem \ref{thm:rank27thm} - part (\ref{algebraic27})]
We begin by appealing to Proposition \ref{bk_bound} which shows that $\dimF \Sel_2(E_{27}/\Q) \le 28$ subject to GRH. However, the root number $\epsilon(E_{27})$ is equal to $-1$, and therefore by Theorem 1.4 in \cite{DD}, we know that $\dimF \Sel_2(E_{27}/\Q)$ is odd. Combined with the fact that $E_{27}$ is known to have at least 27 independent points, this shows that subject to GRH, $\dimF \Sel_2(E_{27}/\Q)$ and therefore the rank of $E_{27}$ are equal to 27.
\end{proof}

\appendix

\section{Proof of Theorem \ref{other_curves}} \label{proof_other_curves}


Theorem \ref{other_curves} is proved in a manner similar to part (\ref{algebraic}) of Theorem \ref{mainthm}, where bounding the $2$-rank of the class group of a cubic field was essential. We found that the existing sieving machinery in \textbf{magma} was sufficient to determine the $2$-Selmer ranks $\dimF \textnormal{Sel}_2(E_r/\Q)$ subject to GRH for $r \in \{20,21,22,23,24\}$. As was the case for $K_{28}$, we use Julia reduction of binary cubic forms to find small defining polynomials for the cubic subfields $K_r$ of $\Q(E_r[2])$. These reduced defining polynomials are listed in Table \ref{tab:def_polys}.

\begin{table}[h]
\def\arraystretch{1.25}
\begin{tabular}{|c|c|}
\hline
$r$ & $f(x)$  \\
\hline
$20$ & $\begin{aligned}13370149617006967x^3 &+36323790822192190x^2\\ 
									&+97698281640159313x - 102297590541619200 \end{aligned}$\\
\hline
$21$ &  $\begin{aligned}274654350297600x^3 &-1624392373464273559x^2\\  
									&-9371598016369119418702x + 6162113868013558026402675 \end{aligned}$\\ 
\hline
$22$ & $\begin{aligned}6142990220640x^3 &+204976117420509373x^2\\ 
								&-169253519238896688671x - 628110960931737938720390\end{aligned}$\\
\hline
$23$ & $\begin{aligned}59865403640328000x^3 &+30357716218004835541x^2\\  
									&-14206611767334834785x + 3031944233345318784207\end{aligned}$ \\
\hline
$24$ & $\begin{aligned} 70256883874320x^3 &+75608696284455934477x^2 \\ 
									&-214624301781108927172690x\\
									& -25666999271392112689637803778\end{aligned}$ \\
\hline
   
\end{tabular}
\vspace{0.2in}
\caption{Defining Polynomials for $K_r$ }
\label{tab:def_polys}
\end{table}
For each of the fields $K_r$, we choose a factor base of degree one primes with norm below the Belabas bound described in Section \ref{subsec:fbbounds}. 
These bounds along with the Bach bound for each $K_r$ are given in Table \ref{tab:prime_bounds}.

\begin{table}[h]
\renewcommand{\arraystretch}{1.25}
\begin{tabular}{|c|c|c|}
\hline
$r $ & \textrm{Bach Bound}  & \textrm{Belabas Bound} \\
\hline
$20$ & $295,854$ & $29,585$ \\
\hline
$21$ & $419,613$ & $55,948$ \\
\hline
$22$ & $371,338$ & $37,133$ \\
\hline
$23$ & $412,632$ & $48,140$ \\
\hline
$24$ & $500,045$ & $66,672$ \\
\hline
$27$ & $908,397$ & $143,829$ \\
\hline
$28$ & $1,202,639$ & $200,439$ \\ 
\hline
   
\end{tabular}
\vspace{0.2in}
\caption{Primes Bounds for $K_r$ }
\label{tab:prime_bounds}

\end{table}

The sieving for the class groups of each $K_r$ was completed using the number field sieve machinery implemented in \textbf{magma} under the \texttt{NFSProcess} command. Each sieve problem was sufficiently small that it could be run overnight on a single CPU. The sieve jobs were all sufficiently small that we did not make any attempt to choose optimal (or even particularly good) sieve regions.

Table \ref{tab:ranks} gives the upper bound for $g(E_{r}) = \dim_{\mathbb{F}_2} Cl(K_r)[2]$ for each $K_r$, along with the values of all of the terms appearing in Proposition \ref{bk_bound} and the global root number $\epsilon(E_r)$ for each of the curves $E_r$ in Theorem \ref{other_curves}.

\begin{table}[h]
\begin{tabular}{|c|c|c|c|c|c|c|}
\hline
$r$ & $g(E_{r}) \leq^*$ & $u(E_r)$ & $n(E_r)$ & $\epsilon(E_r)$ & $\dimF \textnormal{Sel}_2(E_r/\Q) \leq^*$
\\
\hline
$20$  & $15$ & $1$ & $5$ & $+1$ & $20$ \\
\hline
$21$ & $14$ & $2$ & $5$ & $-1$&  $21$ \\
\hline
$22$  & $16$ & $2$ & $4$ &$+1$& $22$ \\
\hline
$23$  & $15$ & $1$ & $8$ &$-1$& $23$ \\
\hline
$24$  & $16$ & $2$ & $7$ & $+1$& $24$ \\
\hline
$27$  & $22$ & $1$  & $5$  & $-1$ & $27$ \\
\hline
$28$ & $20$ & $2$ & $6$ & $+1$& $28$ \\
\hline
\end{tabular}
\vspace{0.2in}
\caption{Calculation of $\dimF \textnormal{Sel}_2(E_r/\Q)$. Bounds denoted with $\leq^*$ depend on GRH.}
\label{tab:ranks}
\end{table}
\begin{proof}[Proof of Theorem \ref{other_curves}]
The second column in Table \ref{tab:ranks} gives upper bounds on the $2$-ranks of the class group $Cl(K_r)$ conditional on GRH. The upper bounds on $\dimF \textnormal{Sel}_2(E_r/\Q)$ are obtained by combining Proposition \ref{bk_bound} with a result of Dokchitser and Dokchitser \cite[Theorem 1.4]{DD}, which ensures that $\epsilon(E) = (-1)^{s(E)}$ in terms of the difference $s(E) = \dimF \textnormal{Sel}_2(E/\Q) - \dimF E(\Q)[2]$. In each case, we see that GRH implies the $2$-Selmer rank $\dimF \textnormal{Sel}_2(E_r/\Q)$ is at most $r$. Since each curve $E_r$ is known to have trivial torsion subgroup and Mordell-Weil rank at least $r$, we conclude that $E_r(\Q) \simeq \Z^r$ subject to GRH.
\end{proof}

\section{Generators for $E_{27}(\Q)$}\label{app::gensE27}

The following is a list of $27$ independent points on the elliptic curve $E_{27}$ given by equation (\ref{elkies27}) above:
%
%
%
\begin{gather*}
\begin{flalign*}
 &(3767967516008165080365044, 2389736302094908158004099904947501190),&  \\  
  &(6870254134405565034404108, 10187524517965617942800545361683736678),&  \\
  &(3887185284020449623939380, 2077020998301366905747533719381033926),&\\
  &(4704247833799635063001076, 2048360739972031724784820678863578246),&\\
  &(4126561570009022393013236, 1587663907962563318996362180056025478),& \\
&(4589477829219012602846900, 1774818405716699582839388275297252934),&\\  &(1744288391661626065189796, 8377495495389391047035879698795823126),& \\
  &(375965292932773063399988, 11878746522289663117790823052090948358),&\\
  &(-4430058725939313297140384, 17935065674772418581237173320631279206),&\\
  &(46029381695079838296565796, 308418721198583803941973238472690797126),&\\
  &(5015368619774521542769364, 2987769291318668561101046595511063430),&\\
  &(55141979583089031946559900, 405905110011451276640435700460385551166),&\\
  &(2703830808220294466353748, 5587793124284970779400186615144247334),&\\
  &(3412724629872318338319668, 3426156011058008602456511184805561094),&\\
  &(272723117214107051072886140, 4502171870151657762741942725666306991014),&\\
  &(4732850534022088572670964, 2124602225002897987491873188898646406),&\\
  &(19225480790209113087907256, 78725996092378368618479740248297817478),&\\
  &(5213267756598937117846508, 3666105463387143768198032469471386414),&\\
  &(-4503215618194252049902522, 17926532987110694852440283715314002874),&\\
  &(10358928712485769814651816, 26398450763063898266637186797421380678),&\\
  &(6560446866541184312028656, 8894515448962144734398280820434671978),& \\ &(4667249764662401626929236, 1954092716090144351072720616414325286),&
\end{flalign*}
\end{gather*}
\begin{gather*}
\begin{flalign*}  
  &(3131745787384349113625300, 4283649283716227803355987840842617734),& \\
  &(243907731994687263474127628, 3807478665185691587984635270031859346574),& \\
  &(110171466072672245507182388, 1153803508275547153736593941741941166854),&\\
  &(2631452133741740392491152, 5805818938673165314161211507146370634),&\\
  &(2398961346477899287733092916, 117498623151243646059583140149253976390406). &
\end{flalign*}
\end{gather*}

\end{document}